\documentclass{article}
\usepackage{blindtext}
\usepackage[T1]{fontenc}
\usepackage[utf8]{inputenc}
\usepackage[english]{babel}
\usepackage{amsmath}
\usepackage{amsfonts}
\usepackage{amssymb}
\usepackage{graphicx}
\usepackage{fancyhdr}
\usepackage{tikz-cd}
\usepackage{amsthm}
\usepackage{MnSymbol}
\usepackage{float}
\usepackage{mathrsfs}
\usepackage{xcolor}
\usepackage[toc]{appendix}
\usepackage{enumitem}
\usepackage[linesnumbered, ruled, vlined]{algorithm2e}
\usepackage{rotating}
\usepackage{fancybox}
\usepackage{extpfeil}
\usepackage{multicol}
\usepackage{multirow}
\usepackage{hhline}
\usepackage{wrapfig}
\usepackage{geometry}
\usepackage{lipsum}
\usepackage{tikz}
\usetikzlibrary{graphs,decorations.pathmorphing,decorations.markings,babel}
\usepackage[colorlinks=true, allcolors=blue]{hyperref}
\usepackage[capitalize, nameinlink]{cleveref}

\crefname{Diagram}{Diagram}{Diagrams}
\crefname{subsec}{Subsection}{Subsections}
\crefname{defin}{Definition}{Defintions}
\crefname{rem}{Remark}{Remarks}

\newtheorem{theo}{Theorem}[section]

\newtheorem{lem}[theo]{Lemma}
\newtheorem{propo}[theo]{Proposition}
\newtheorem{coropro}[theo]{Corollary}

{%
\theoremstyle{definition}
\newtheorem{defin}[theo]{Definition}
\theoremstyle{definition}
\newtheorem{rem}[theo]{Remark}
\theoremstyle{definition}

}

\SetKwInput{Input}{Input} 
\SetKwInput{Output}{Output}

\newcommand\blfootnote[1]{%
  \begingroup
  \renewcommand\thefootnote{}\footnote{#1}%
  \addtocounter{footnote}{-1}%
  \endgroup
}

\title{Invariant Grassmannians and a K3 surface with an action of order 192*2}
\author{Stevell Muller}
\date{}
\begin{document}
\maketitle

\begin{abstract}
Given a complex vector space $V$ of finite dimension, its Grassmannian variety parametrizes all subspaces of $V$ of a given dimension. Similarly, if a finite group $G$ acts on $V$, its invariant Grassmannian parametrizes all the $G-$invariant subspaces of $V$ of a given dimension. Based on this fact, we develop an algorithm for finding equations of $G-$invariant projective varieties arising as an intersection of hypersurfaces of the same degree.

We apply the algorithm to find equations describing a polarized K3 surface with a faithful action of $T_{192}\rtimes \mu_2$ and some further symmetric K3 surfaces with a degree 8 polarization.
\end{abstract}

\tableofcontents 

\section*{Introduction}\blfootnote{\hspace*{-5.4ex} Gefördert durch die Deutsche Forschungsgemeinschaft (DFG) – Projektnummer 286237555 – TRR 195.\\
 Funded by the Deutsche Forschungsgemeinschaft (DFG, German Research Foundation) – Project-ID 286237555 – TRR 195.}
\blfootnote{Author: Muller Stevell, Universität des Saarlandes, muller@math.uni-sb.de}

A \emph{K3 surface} is a simply-connected, compact, complex manifold $S$ admitting a nowhere vanishing holomorphic symplectic $2-$form $\sigma_S\in H^{2,0}(S)$, which is unique up to scaling. Given a projective K3 surface $S$ and a finite group $G\leq\text{Aut}(S)$, the action of $G$ on $H^{2, 0}(S)$ induces a short exact sequence
\[ 1\to G_s\to G\to \mu_n\to 1\]
where $G_s\trianglelefteq G$ is the normal subgroup of so-called \emph{symplectic} automorphisms. The classification of such finite groups of automorphisms of K3 surfaces is today well-understood. In \cite{nik80}, Nikulin classifies finite symplectic actions on K3 surfaces which are abelian. This was later extended by Mukai \cite{muk88} who shows that for a K3 surface $S$, a finite subgroup of symplectic automorphisms $G_s\leq \textnormal{Aut}(S)$ embeds into one among 11 maximal groups. A simpler proof of Mukai's result is given by Kond\=o \cite{kon98}, using a relation with automorphisms of the Niemeier lattices. A list of such groups of symplectic automorphisms was recomputed by Xiao \cite{xia96}, together with combinatorial data about the fixed points of the corresponding action on the surface. Further transcendental data associated to the action of these groups on the intersection form on the second integral cohomology were determined by Hashimoto in \cite{has12}. On the other side of the spectrum, purely non-symplectic automorphisms (non-symplectic automorphisms without any non-trivial symplectic iterates) have been studied and classified for all orders. The literature on the subject is quite rich: we refer to \cite{st} for an account on this problem.\bigskip

In order to extend the classification of finite symplectic actions to a classification of finite \emph{mixed actions}, Brandhorst and Hashimoto describe in \cite{bh20} a lattice-theoretic approach which they apply to the 11 maximal groups classified by Mukai \cite{muk88}. They classify 42 isomorphism classes of pairs $(S, G)$ where $S$ is a projective K3 surface and $G\leq\text{Aut}(S)$ is such that the normal subgroup $G_s\trianglelefteq G$ of symplectic automorphisms is a proper subgroup, and it is isomorphic to one of the 11 aforementioned maximal groups. Such pairs $(S, G)$ come equipped with a canonical $G$-invariant \emph{polarization} $L$, which is a primitive ample line bundle on $S$. For each case such a surface $S$ has maximal Picard rank. The authors moreover exhibit 25 cases for which an explicit projective model of $S$, by means of equations, is known. In particular, all of the triples $(S,G, G_s)$ of degree $c_1(L)^2\leq 10$ have been treated, except one where the associated symplectic action has order 192 (case 77b). Following the notation in \cite{bh20} (except the polarization that we denote "$L$" here and not "$l$", and the transcendental lattice which we write "$T_S$"), we display in \cref{tab:my_labelbis} some information about this isomorphism class of K3 surfaces (where $T_{192}$ is one of the maximal subgroups of symplectic automorphisms classified in \cite{muk88}).\bigskip
\renewcommand{\arraystretch}{1.25}
\begin{table}[t!]
    \centering
    \begin{tabular}{|c|c|c|c|c|c|c||c|}
        \hline 
         case&$G_s$&$\Lambda_{K3}^{G_s}$&$\textrm{SO}(\Lambda_{K3}^{G_s})$&$ G/G_s$&$c_1(L)^2$&$T_S$&$G$  \\
         \hline
         77b&$T_{192}$&$\begin{pmatrix}4&0&0\\0&8&4\\0&4&8\end{pmatrix}$&$D_6$&$\mu_2$&8&$\begin{pmatrix}4&0\\0&24\end{pmatrix}$&$\begin{array}{c}T_{192}\rtimes \mu_2\\\rm{GAP\,Id\;} [384,5602]\end{array}$\\
         \hline
    \end{tabular}
    \caption{Specification for the triple $(S,G, G_s)$ considered in this paper}
    \label{tab:my_labelbis}
\end{table}

Knowing the transcendental lattice $T_S$ of such a projective K3 surface $S$ already gives one some known properties about $S$. It is a Kummer surface and it is the unique K3 surface of degree 8 with a faithful symplectic action of $T_{192}$. Moreover, the surface $S$ is the \emph{Barth-Bauer octic} with the second largest number (160) of conics (see \cite[Theorem 4.2, Example 4.3]{deg22}). However no equations describing $S$ are known. We aim to prove the following:

\begin{theo}\label{theo:mainth}
The polarized K3 surface $(S, L)$ corresponding to the case 77b in \cite{bh20} admits a projective model in $\mathbb{P}^5_{\mathbb{C}}$ given by
\[ S\colon\left\{\begin{array}{rlrlrlrlrlr} i x_0x_1&+& x_0x_2&+& x_1x_3&+&i x_2x_3&+& x_5^2& = 0\\
i x_0x_1&-& x_0x_2&-& x_1x_3&+&i x_2x_3&+& x_4^2& = 0\\
- x_0x_3&&&-&x_1x_2&&&-& x_4x_5& = 0
\end{array}\right..\]
It admits a maximal symplectic action of $T_{192}$ and is invariant under the linear action of $G := T_{192}\rtimes \mu_2$ on $\mathbb{P}^5_{\mathbb{C}}$ given by
\begin{footnotesize}
\begin{align*}
&\begin{array}{ccc}
\sigma_1 = \begin{pmatrix}0&-1&0&0&0&0\\-1&0&0&0&0&0\\0&0&0&-1&0&0\\0&0&-1&0&0&0\\0&0&0&0&1&0\\0&0&0&0&0&1\end{pmatrix}& \sigma_2 = \frac{1}{2}\begin{pmatrix}0&0&0&-2&0&0\\0&2&0&0&0&0\\2&0&0&0&0&0\\0&0&-2&0&0&0\\0&0&0&0&z^6-1&-z^6-1\\0&0&0&0&-z^6+1&-z^6-1\end{pmatrix} &\sigma_3= \begin{pmatrix}0&0&1&0&0&0\\0&0&0&-1&0&0\\-1&0&0&0&0&0\\0&1&0&0&0&0\\0&0&0&0&z^6&0\\0&0&0&0&0&-z^6\end{pmatrix}
\end{array}\\
&\begin{array}{cc}
\sigma_4 = \frac{1}{2}\begin{pmatrix}1&-1&-1&-1&0&0\\-1&-1&-1&1&0&0\\-1&-1&1&-1&0&0\\-1&1&-1&-1&0&0\\0&0&0&0&-z^5+z^3+z&-z^5-z^3+z\\0&0&0&0&z^5+z^3-z&z^5-z^3-z\end{pmatrix}&
\sigma_5 = \frac{1}{2}\begin{pmatrix}-1&-1&1&-1&0&0\\-1&-1&-1&1&0&0\\1&-1&-1&-1&0&0\\-1&1&-1&-1&0&0\\0&0&0&0&-2&0\\0&0&0&0&0&-2\end{pmatrix}
\end{array}
\end{align*}
\end{footnotesize}where $T_{192}$ is one of the maximal subgroups of symplectic automorphisms classified in \cite{muk88} and $z$ is a primitive $24-$th root of unity.
\end{theo}

\begin{rem}
From these equations and the description of the group action, one could possibly compute explicitly representatives for the $G$-orbits of 160 conics on $S$, similarly to \cite{bs21,nas22} for the 3 orbits of 800 conics on the $M_{20}$-quartic. Nonetheless, performing such computations appears to be expensive. We do not address this problem in the present paper.
\end{rem}

The main goal of this paper is to develop an algorithm for computing equations describing projective K3 surfaces given as \emph{(complete) intersections} of $t$ hypersurfaces of the same degree $d$ in the projective space $\mathbb{P}^n_{\mathbb{C}}$, with prescribed group of symmetry $G$. The complete intersection case applies to \emph{general} polarized K3 surfaces of degree 2, 4 and 8.


\subsection*{Plan of the paper}
\hspace{\parindent}In \cref{sec:prelem} we review some results about linear representations of finite groups as well as their projective representations. 

The core of this paper is found in \cref{sec:para} where we discuss an algorithmic way of classifying subrepresentations of a given representation of a finite group over an algebraically closed field of characteristic zero. We use in particular that one can parametrize, as for Grassmannian varieties in the case of vector spaces, subrepresentations of a given dimension and given character of a higher dimensional representation of a finite group. 

Then, we show in \cref{sec:2} how one can use this systematic study of finite group representations to parametrize defining ideals of intersections of hypersurfaces of the same degree preserved under an action of a finite group on their ambient projective space. 

Finally, in \cref{sec:sec4} we comment on some geometric features arising from the K3 surface described in \cref{theo:mainth}.

All the algorithms written for the purpose of this paper have been implemented in the computer algebra system Oscar \cite{Osc}, written in Julia \cite{Julia-2017}.

\section*{Acknowledgements}
\hspace{\parindent}The author would like to thank Simon Brandhorst for suggesting him the subject, for the useful discussions and all the help provided. The author would also like to thank Max Horn and Matthias Zach for helping to optimize the implementation of the algorithms used for the purpose of this paper, and Tommy Hofmann for the precious comments made on a draft of this paper. The author is grateful to Benedetta Piroddi and Enrico Fatighenti for the discussions regarding the content of \cref{sec:sec4}, and to anonymous referees for their corrections and suggestions which improved the quality of the paper.

\section{Preliminaries on representation theory}\label{sec:prelem}
\hspace{\parindent}  In this section we recall some facts about representation theory and character theory for the readers who are not familiar with these notions. In particular, we fix some definitions and notations for the rest of the paper, and we state the relevant results we use throughout. In this paper we work over algebraically closed fields of characteristic zero, mainly $\mathbb{C}$, all groups are supposed to be finite and all vector spaces are of finite dimension.

\subsection{Linear representations and group algebra modules}
\hspace{\parindent}  The definitions and results of this subsection can be found in any classical book about representation theory of finite groups. We refer, for instance, to the book \cite{rpth}.\\

Let $K$ be an algebraically closed field of characteristic zero and let $E$ be a finite group. By Maschke's theorem \cite[Theorem 3.1]{rpth}, the group algebra $KE$ is \emph{semisimple}, that is, all of its modules are semisimple and therefore can be decomposed as the direct sum of simple submodules. Throughout this paper, we describe $KE-$modules as pairs $(V, \rho)$, where $V$ is a finite-dimensional $K-$vector space and $\rho$ is a \emph{$K-$linear representation of $E$ on $V$}, that is $\rho$ is a homomorphism
\begin{equation*}
    \rho\colon E\to\text{GL}(V).
\end{equation*}

Given two $KE-$modules $M = (V, \rho)$ and $M' = (V', \rho')$, we say that $M$ and $M'$ are \emph{equivalent}, and we write $M \cong M'$, if there exists an invertible $K-$linear map $\mathcal{L}\colon V\to V'$ such that for all $e\in E$,
\[ \rho^{\mathcal{L}}(e) := \mathcal{L}\circ\rho(e)\circ\mathcal{L}^{-1}.\]
If $V = V'$, we also say that $\rho$ and $\rho'$ are themselves \emph{equivalent}.

\begin{rem}\label{rem:somrem}
By the Krull-Schmidt theorem \cite[Theorem 2.19]{rpth}, if a $KE-$module is semisimple then its decomposition into a direct sum of simple submodules is unique up to equivalence and order of the summands. Moreover, according to \cite[Corollary (2.5)]{isa76}, the equivalence classes of simple $KE-$modules correspond bijectively to the conjugacy classes of $E$. 
\end{rem}

If $M = (V, \rho)$ is a $KE-$module, then by Maschke's theorem, one can write 
\begin{equation*}
    M = \bigoplus_{i=1}^l W_i^{f_i} 
\end{equation*}
where the $W_i$'s are pairwise non equivalent simple $KE-$modules.
We call this decomposition an \emph{isotypical decomposition} of $M$. For all $1\leq i\leq l$, we call the summand $W_i^{f_i}$ an \emph{isotypical component} of $M$. It is itself a $KE-$module which we say to be \emph{isotypical of weight} $\textnormal{dim}_K(W_i)$ (to be understood, the $K-$dimension of the underlying vector space). Although the decomposition of $M$ into the sum of its isotypical components is unique, the decomposition of each $W_i^{f_i}$ into a sum of simple modules is unique only up to equivalence.\\

We conclude this subsection by stating one key result we use several times in this paper.
\begin{theo}[Schur's lemma; \cite{rpth}, Proposition 1.16, Corollary 1.17]\label{lem:schur}
Let $M= W^{\oplus t}$ and $M'= W'^{\oplus t'}$ be two isotypical $KE-$modules, where $W$ and $W'$ are simple. Then, under the assumption that $K$ is algebraically closed, one has
\begin{equation*}
    \textnormal{Hom}_{KE}(M, M') \simeq \left\{\begin{array}{lll}\textnormal{Mat}_{t,t'}(K)&\text{if}&W\cong W'\\
    0&\text{else}&
    \end{array}\right.
\end{equation*}
where $\textnormal{Mat}_{t,t'}(K)$ denotes the set of $t-$by$-t'$ matrices with entries in $K$.
In particular, the $KE-$ automorphism group of a simple $KE-$module can be identified with $K^{\times}$.
\end{theo}

\subsection{Characters of representations}
\hspace{\parindent}  To any $K-$linear representation of $E$ on a vector space $V$, one can associate a so-called \emph{character}. These characters encode most of the information one needs to study $KE-$modules. From a computational point of view, characters of finite groups are an efficient tool to work with and allow us to simplify many algorithmic methods in representation theory. To read more about characters we refer to \cite{isa76}. \\

Again, let $K$ be an algebraically closed field of characteristic 0, let $E$ be a finite group, and let $M = (V,\rho)$ be a $KE-$module. We define the \emph{$K-$character} $\chi_{M}$ of $M$ to be the mapping
\begin{equation*}
    \chi_{M}\colon E\to K,\; e\mapsto \text{Tr}(\rho(e)).
\end{equation*}
We say also that $M$ \emph{affords} $\chi_{M}$ and that $\chi_{M}$ is \emph{afforded by} $M$. One notes that $\chi_{M}(1_E) = \text{dim}_K(V)$ and $\chi_{M}$ is constant on each conjugacy class of $E$. More generally, $K-$characters of $E$ are a special case of what we call \emph{class functions} on $E$. Moreover, for any $K-$character $\chi$ of $E$, there is a $KE-$module $M$ such that $\chi = \chi_{M}$. We define sum and product of $K-$characters of $E$  as pointwise sum and product of their respective images in $K$. So for instance, if $\chi$ and $\chi'$ are two $K-$characters of $E$ afforded by $M$ and $M'$ respectively, then $\chi+\chi'$ is afforded by $M \oplus M'$ and vice-versa. A $K-$character $\chi$ of $E$ is said to be \emph{simple}, or \emph{irreducible}, if $\chi$ cannot be written non-trivially as sum of other $K-$characters of $E$. 

\begin{propo}[\cite{isa76}, Corollary (2.5)]
The number of simple $K-$characters of $E$ is equal to the number of conjugacy classes of $E$ (recall that $E$ is a finite group here). In particular, simple $K-$characters of $E$ are afforded by simple $KE-$modules.
\end{propo}

\begin{propo}[\cite{isa76}, Corollary (2.9)]\label{propo:samechar}
  Two $KE-$modules $M$ and $M'$ are equivalent if and only if they afford the same $K-$character of $E$.
\end{propo}

We define the \emph{degree} of a $K-$character $\chi$ of $E$ as $\chi(1_E)$. For all $n\geq 1$, we denote $\textnormal{Irr}^n_K(E)$ the set of all simple $K-$characters of $E$ of degree $n$, and we define $\textnormal{Irr}_K(E) := \bigsqcup_{n\geq 1}\text{Irr}^n_K(E)$. 
According to \cite[Theorem (2.8)]{isa76}, any $K-$character $\chi$ of $E$ admits a unique decomposition
\begin{equation*}\label{eq:chardec}
    \chi = \sum_{\mu\in\text{Irr}_K(E)}e_{\mu}\mu
\end{equation*}
where $e_{\mu}\in \mathbb{Z}_{\geq 0}$ is called the \emph{multiplicity} of the irreducible character $\mu$ in $\chi$. Given two $K-$characters $\chi = \sum_{\mu\in\text{Irr}_K(E)}e_{\mu}\mu$ and $\chi' = \sum_{\mu\in\text{Irr}_K(E)}e'_{\mu}\mu$ of $E$, we define their \emph{scalar product}
\begin{equation*}
    \langle \chi, \chi' \rangle := \sum_{\mu\in\textnormal{Irr}_K(E)}e_{\mu}e'_{\mu}.
\end{equation*}
In particular, for $\mu\in\textnormal{Irr}_K(E)$ we have that $\langle \mu, \mu \rangle = 1$, and $\langle \chi, \mu \rangle$ is equal to the multiplicity of $\mu$ in $\chi$. If $\chi = \sum_{\mu\in\text{Irr}_K(E)}e_{\mu}\mu$ and $\chi' = \sum_{\mu\in\text{Irr}_K(E)}e'_{\mu}\mu$ are two $K-$characters of $E$ such that $0 \leq e_{\mu} \leq e'_{\mu}$ for all $\mu\in \textnormal{Irr}_K(E)$, then we say that $\chi$ is a \emph{constituent} of $\chi'$.\\

We see that the decomposition of the $K-$character afforded by  a $KE-$module depends only on its isotypical decomposition. We say that a $K-$character $\chi$ of $E$ is \emph{isotypical} if $\chi$ is afforded by an isotypical $KE-$module, i.e. it is a positive multiple of an irreducible $K-$character of $E$.

\subsection{Group actions and projective representations}\label{sec:1} 
\hspace{\parindent}  In this subsection we state some relevant results about projective representations of finite groups. We refer to \cite[Chapter 11]{isa76} for the readers who are not familiar with the notion of projective representations. The key point is that one can relate linear actions of a group $G$ on $\mathbb{P}(\mathbb{C}^{n+1})$ to linear actions of a possibly larger group, called a \emph{Schur cover}, on $\mathbb{C}^{n+1}$.\\

Let $K$ be algebraically closed of characteristic zero. Given a finite group $G$ and a finite-dimensional $K-$vector space $V$, we call a \emph{projective representation} of $G$ on $V$ any homomorphism 
\[ \overline{\rho}\colon G\to \textnormal{PGL}(V).\]
Such a representation is called \emph{faithful} if it is injective. For any group $G$, there exists a finite abelian group $M(G)$ called the \emph{Schur multiplier} of $G$ (see \cite[Definition (11.12)]{isa76}), which can be identified with $H^2(G, K^{\times})$, the second cohomology group of $G$ with coefficients in $K^{\times}$. In \cite{Schur1904}, Schur proves that for any finite group $G$, there exists a group $E$ and an exact sequence
\[\begin{tikzcd}
1\to H\xrightarrow{i} E\xrightarrow{p} G\to 1
\end{tikzcd}\]
such that $H\cong M(G)$, and $i(H) \leq [E, E]\cap Z(E)$ where $[E, E]$ is the derived subgroup of $E$ and $Z(E)$ is its center. This exact sequence satisfies the following: for any projective representation $\overline{\rho}\colon G\to \text{PGL}(V)$ of $G$ on a finite-dimensional $K-$vector space $V$, there exists a linear representation $\rho\colon E\to \text{GL}(V)$ making the following diagram with exact rows commute 
\begin{equation}\label[Diagram]{diag:maindiag}
\begin{tikzcd}
1\arrow[r]&H\arrow[d,"\beta"']\arrow[r,"i"]&E\arrow[d,"\exists\rho"']\arrow[r, "p"]&G\arrow[d, "\overline{\rho}"']\arrow[r]&1\\
1\arrow[r]&K^{\times}\arrow[r,"\cdot\textnormal{Id}_V"']&\textnormal{GL}(V)\arrow[r,"\pi"']&\textnormal{PGL}(V)\arrow[r]&1
\end{tikzcd}.
\end{equation}
Here $\beta$ is induced by the restriction of $\rho$ to $i(H)$. This result is known as Schur's theorem (see \cite[Theorem (11.17)]{isa76}) and we call $p\colon E\twoheadrightarrow G$ (or just $E$) a \emph{Schur cover} of $G$. We moreover refer to $\rho$ as a \emph{$p-$lift} of $\overline{\rho}$ and the latter as the \emph{$p-$reduction} of the former. Schur's theorem allows us to use the results from the theory of linear representations of finite groups to work with projective representations of finite groups. In particular, given a finite group $G$, a Schur cover $p\colon E\twoheadrightarrow G$ and a finite-dimensional $K-$vector space $V$, one can relate a classification of projective representations of $G$ on $V$ to a classification of linear representations of $E$ on $V$.

\begin{rem}
    In general, a Schur cover $E$ of a finite group $G$ is not unique. If $G = [G, G]$ is perfect, then $E$ is unique up to isomorphism. Otherwise, an upper bound on the number of non-isomorphic Schur covers of $G$ can be found in \cite[Theorem 2.5.14]{kar87}.
\end{rem}

\begin{defin}[\cite{isa76}, Definition(1.18); \cite{isa76}, Page 177]\label[defin]{defin:def2}
Let $G$ be a finite group and let $V$ be a finite-dimensional $K-$vector space. Two projective representations $\overline{\rho}, \overline{\rho}'\colon G\to \text{PGL}(V)$ are called \emph{similar} if there exists an automorphism $\mathcal{L}\colon V \to V$ such that, for all $g\in G$,
\[\overline{\mathcal{L}}\circ\overline{\rho}(g) = \overline{\rho}'(g)\circ\overline{\mathcal{L}}\]
where $\overline{\mathcal{L}}\colon \mathbb{P}(V)\to \mathbb{P}(V)$ is induced by $\mathcal{L}$. 
\end{defin}

\begin{lem}[\cite{isa76}, Page 178]\label{lem:lem1}
Let $G$ be a finite group, let $p\colon E\twoheadrightarrow G$ a Schur cover of $G$ and let $V$ a finite-dimensional $K-$vector space. Assume that there are two projective representations $\overline{\rho}, \overline{\rho}'\colon G\to \textnormal{PGL}(V)$ with respective $p-$lifts $\rho, \rho'\colon E\to \textnormal{GL}(V)$ as in \cref{diag:maindiag}. Then $\overline{\rho}$ and $\overline{\rho}'$ are similar if and only if there exists a homomorphism $\epsilon\colon E\to K^{\times}$ such that $\rho$ and $\epsilon\rho'$ are equivalent.
\end{lem}

\begin{proof}
Suppose that $\overline{\rho}$ and $\overline{\rho}'$ are similar and let $\mathcal{L} \in \textnormal{GL}(V)$ be such that, for all $g\in G$,
\[ \overline{\mathcal{L}}\circ\overline{\rho}(g)\circ\overline{\mathcal{L}}^{-1} = \overline{\rho}'(g).\]
By commutativity of \cref{diag:maindiag}, for all $e\in E$, one obtains that

\begin{align*}
    \pi(\mathcal{L}\circ\rho(e)\circ\mathcal{L}^{-1}) = \overline{\mathcal{L}}\circ\underbrace{(\pi(\rho(e)))}_{=\overline{\rho}(p(e))}\circ\overline{\mathcal{L}}^{-1}    &= \overline{\rho}'(p(e)) = \pi(\rho'(e)).
\end{align*}
Hence, for all $e\in E$, there exists a unique scalar $\epsilon(e)\in K^\times$ such that
$\mathcal{L}\circ\rho(e)\circ\mathcal{L}^{-1} = \epsilon(e)\rho'(e).$
By straightforward computations, one can show that the previous assignment $\epsilon\colon E\to K^{\times}$ is a group homomorphism, and $\rho$ and $\epsilon\rho'$ are equivalent. Now suppose that there exist a homomorphism $\epsilon\colon E\to K^{\times}$ and $\mathcal{L}\in\textrm{GL}(V)$ such that, for all $e\in E$,
\[ \mathcal{L}\circ\rho(e)\circ\mathcal{L}^{-1} = \epsilon(e)\rho'(e).\]
By commutativity of \cref{diag:maindiag} and surjectivity of $p$, it clear that $\overline{\rho}$ and $\overline{\rho}'$ are similar
\end{proof}

Note that, in the context of \cref{lem:lem1}, given a linear representation $\rho$ of $E$ on $V$ whose restriction to $M(G) = \ker(p)$ maps to $K^{\times}\text{Id}_V$, one can always define a projective representation of $G$ on $V$ which makes \cref{diag:maindiag} commute. Indeed, we can define a $p-$reduction of $\rho$ as $\overline{\rho} := \pi\circ\rho\circ s$ where $s$ is any section of $p$ that maps $1_G$ to $1_E$ (it can be easily shown that this definition does not depend on the choice of $s$). In this case, we say that $\rho$ is \emph{$p-$projective}.

\section{Parametrizing submodules of a given group algebra module}\label{sec:para}
\hspace{\parindent}  Let $K$ be an algebraically closed field of characteristic zero, let $E$ a finite group and let $M = (V, \rho)$ a $KE-$module. In this section we show that the set parametrizing the $KE-$submodules of $M$ of a given dimension is a projective variety, and that its irreducible components are rational.

\subsection{Invariant Grassmannians}\label{subsec:symgr}
\hspace{\parindent} Let $M = (V,\rho)$ be a $KE-$module, and let $n := \textnormal{dim}_K(V)$. For all $1\leq t\leq n$, we define $\text{Gr}(t, M)$ to be the set of $t-$dimensional $KE-$submodules of $M$. Using iteratively an argument from \cite[Theorem 5.32]{shuffle}, we have that $\text{Gr}(t, M)$ is a closed subvariety of the ordinary Grassmannian variety $\textnormal{Gr}(t,V)\subseteq \mathbb{P}(\bigwedge^t V)$. In general, $\text{Gr}(t, M)$ is not irreducible, and we give two ways to decompose it: we use the first one computationally to parametrize all $t-$dimensional submodules of $M$, and the second one can be used to compute the defining ideal of $\textnormal{Gr}(t, M)$ in $\mathbb{P}(\bigwedge^t V)$.\\

Let $\chi$ be the $K-$character of $E$ afforded by $M$. 

\begin{defin}
    For all $1\leq t\leq \chi(1_E)$, we define $C_{\chi}(t)$ to be the set of all $K-$characters of $E$ of degree $t$ which are a consistuent of $\chi$.
\end{defin}

For $1\leq t\leq \chi(1_E)$, each $\eta\in C_{\chi}(t)$ defines an equivalence class of $t-$dimensional submodules of $M$. For $\eta\in C_{\chi}(t)$, let $N_{\eta}$ be a $KE-$module affording $\eta$ and define
\[\text{Gr}(\eta, M) := \textnormal{Hom}^0_{KE}(N_{\eta}, M)/\textnormal{Aut}_{KE}(N_{\eta})\]
where $\textnormal{Hom}^0_{KE}(N_{\eta}, M)$ denotes the set of all injective $KE-$module homomorphisms from $N_\eta$ to $M$. By \cref{propo:samechar}, this definition does not depend on the choice of $N_{\eta}$, and one can see that $\textnormal{Gr}(\eta, M)$ corresponds to the set of $t-$dimensional submodules of $M$ affording $\eta$. Therefore, as a set, 
\[ \text{Gr}(t, M) = \bigsqcup_{\eta\in C_{\chi}(t)}\text{Gr}(\eta, M).\]

Now, let $M =\bigoplus_{\mu\in\textnormal{Irr}_K(E)}W_{\mu}^{\oplus f_{\mu}}$ be an isotypical decomposition of $M$ where for all $\mu\in\text{Irr}_K(E)$, $W_\mu$ affords $\mu$, and let $\eta = \sum_{\mu\in \textnormal{Irr}_K(E)}e_{\mu}\mu\in C_{\chi}(t)$. Then, the $KE-$module $\bigoplus_{\mu\in\textnormal{Irr}_K(E)}W_{\mu}^{\oplus e_{\mu}}$ affords $\eta$, and we have that

\[ \text{Gr}(\eta, M) = \text{Hom}^0_{KE}\left(\bigoplus_{\mu\in\textnormal{Irr}_K(E)}W_{\mu}^{\oplus e_{\mu}}, \bigoplus_{\mu\in\textnormal{Irr}_K(E)}W_{\mu}^{\oplus f_{\mu}}\right)\Bigg\slash\text{Aut}_{KE}\left(\bigoplus_{\mu\in\textnormal{Irr}_K(E)}W_{\mu}^{\oplus e_{\mu}}\right).\]
Using Schur's Lemma (\cref{lem:schur}), we compute
\begin{equation}\label{eq:isofac}
\text{Gr}(\eta, M) \simeq \prod_{\mu\in\text{Irr}_K(E),\;e_{\mu}\neq 0}\textnormal{Hom}^0_{KE}(W_\mu^{e_\mu}, W_\mu^{f_\mu})/\textnormal{Aut}_{KE}(W_\mu^{e_\mu}) = \prod_{\mu\in\text{Irr}_K(E),\;e_{\mu}\neq 0}\text{Gr}(e_{\mu}\mu, W_{\mu}^{f_{\mu}}).
\end{equation}

We are therefore reduced to understanding $KE-$submodules of isotypical $KE-$modules.

\subsection{Isotypical modules and Gauss elimination}
\hspace{\parindent} The goal of this subsection is to bring a constructive approach to the proof of the existence of a Gauss elimination theorem for isotypical $KE-$modules. Thanks to Schur's lemma, this immediatly follows from the fact that Grassmannians of isotypical submodules are the same as ordinary Grassmannians.\\

Let $M = W^{\oplus t}$ be isotypical of weight $n$ and $K-$dimension $tn$ with $W$ simple. We want to study the set $\text{Gr}(k, M)$ of equivalence classes of $k-$dimensional $KE-$submodules of $M$, for some $k\geq 1$. As a first remark, since the character $\chi$ afforded by $M$ is isotypical, it can be written as $\chi = t\mu$ where $\mu\in\text{Irr}_K(E)$ is afforded by $W$. So, in particular, since $\mu$ is of degree $n$, 
\[ \text{Gr}(k, M) = \left\{\begin{array}{ll}
\text{Gr}(r\mu, M)&\text{if $k = rn$ for some $1\leq r\leq t$}\\
\emptyset&\text{otherwise}\end{array}\right..\]

\begin{theo}[Gauss elimination]\label{theo:ge}
Let $M = W^{\oplus t}$ be an isotypical $KE-$module of weight $n$ and $K-$dimension $tn$. Then for all $1 \leq r\leq t$, the set $\textnormal{Gr}(rn, M)$ of $KE-$submodules of $M$ of dimension $rn$ can be identified with the ordinary Grassmannian $\textnormal{Gr}(r, t)$.
\end{theo}

\begin{proof}
    Let $1\leq r\leq t$ and let $\chi = t\mu$, with $\mu\in\text{Irr}_K(E)$, be the character afforded by $M$. Using the fact that $\text{Gr}(rn, M) = \text{Gr}(r\mu, M)$, since $M$ is isotypical, we have that
    \[\text{Gr}(rn, M) = \text{Hom}^0_{KE}(W^{\oplus r}, W^{\oplus t})\big\slash \text{Aut}_{KE}(W^{\oplus r}).\]
    Fixing a basis of the underlying $K-$vector space of $W$, Schur's Lemma (\cref{lem:schur}) gives us then that
    \[\text{Gr}(rn, M) \simeq \text{Mat}^0_{r, t}(K)\big\slash\text{GL}_r(K) = \text{Gr}(r, t)\]
    where $\text{Mat}^0_{r, t}(K)$ denotes the set of full rank $r-$by$-t$ matrices with entries in $K$.
\end{proof}

The important point of the proof of \cref{theo:ge} is the following: given a $KE-$module $M = \bigoplus_{\mu\in\text{Irr}_K(E)} W_{\mu}^{f_{\mu}}$, one can algorithmically compute a basis of $\text{Hom}_{KE}(W_{\mu}, M)$ for all $\mu$ such that $f_{\mu} \neq 0$. In fact, denoting $M = (V, \rho)$ and $W_{\mu} = (V', \rho')$, and fixing respective $K-$bases $B$ and $B'$ of $V$ and $V'$, $\text{Hom}_{KE}(W_{\mu}, M)$ corresponds to the set $\text{M}_{\rho, \rho'}(B, B')$ of $\text{dim}_K(V)\times\text{dim}_K(V')$ matrices $P$ such that, for all $e\in E$,
\[ \rho(e)P = P\rho'(e).\]
(Here we identify $\text{GL}(V)$ and $\text{GL}(V')$ with the respective groups of invertible matrices using the fixed bases $B$ and $B'$) This is a $K-$vector space of finite dimension $f_{\mu}$ and any non-zero matrix in $\text{M}_{\rho, \rho'}(B, B')$ is of full rank (since $W_{\mu}$ is simple). A basis of this vector space can be computed using for instance an algorithm of \cite[Theorem 2]{algomod}. In particular, for all $1\leq e \leq f_{\mu}$, one has that any embedding $W_{\mu}^e\hookrightarrow M$ corresponds to the choice of an $e-$space in $\text{M}_{\rho, \rho'}(B, B')$.

\subsection{Rationality and irreducible components}\label{subsec:moduli}
\hspace*{\parindent} In this subsection we show that the space $\text{Gr}(\eta, M)$ of $KE-$submodules with a given character $\eta$ of any $KE-$module $M$ is a rational projective variety. Once again, thanks to Schur's lemma, this is a direct consequence of the fact that $\text{Gr}(\eta, M)$ is a finite product of ordinary Grassmannians.\\

Let $M$ be a $KE-$module and let us denote $\chi$ the character afforded by $M$.
\begin{theo}\label{theo:decgra}
For all $1\leq t\leq \chi(1_E)$ and for all $\eta\in C_{\chi}(t)$, the space $\textnormal{Gr}(\eta, M)$ is a rational subvariety of $\textnormal{Gr}(t, M)$ of dimension $\langle \eta, \chi-\eta\rangle$. In particular, $\left\{\textnormal{Gr}(\eta, M)\right\}_{\eta\in C_{\chi}(t)}$ is the set of irreducible components of $\textnormal{Gr}(t, M)$.
\end{theo}

\begin{proof}
Using \cref{eq:isofac} and \cref{theo:ge}, one can endow $\text{Gr}(\eta, M)$ with a scheme structure as a direct product of rational varieties, allowing us to see it as rational subvariety of $\text{Gr}(t, M)$. More presicely, if $M = \bigoplus_{\mu\in\textnormal{Irr}_K(E)}W_{\mu}^{\oplus f_{\mu}}$, we set that
\[\text{Gr}(\eta, M) \simeq \prod_{\mu\in\text{Irr}_K(E),\;e_{\mu}\neq 0}\text{Gr}(e_{\mu}, f_{\mu})\]
as a projective variety. It is known that the ordinary Grassmannian $\textnormal{Gr}(r,t)$ has dimension $r(t-r)$, as a projective variety. Thus, one deduces that
\[\textnormal{dim}(\textnormal{Gr}(\eta, M)) =\sum_{\mu\in\textnormal{Irr}_K(E),\;e_{\mu}\neq 0} e_\mu(f_\mu-e_\mu) = \sum_{\mu\in\textnormal{Irr}_K(E),\;e_{\mu}\neq 0}\langle e_\mu\mu, f_\mu\mu - e_\mu\mu\rangle = \langle \eta, \chi-\eta\rangle\]
by the orthogonality relations between simple characters with respect to the scalar product.
\end{proof}

\cref{theo:decgra} offers a feasible way to parametrising $t-$dimensional submodules of a given $KE-$ module $M$. Indeed, for all $t-$dimensional constituent $\eta$ of the character $\chi$ of $M$, one may use \cref{theo:ge} and \cref{theo:decgra} to construct a concrete parametrisation of $\text{Gr}(\eta, M)$.

\subsection{Determinant character}
\hspace*{\parindent} We state at the beginning of \cref{subsec:symgr} that $\textnormal{Gr}(t, M)$ is a closed subvariety of $\textnormal{Gr}(t ,V)\subseteq \mathbb{P}(\bigwedge^t V)$, where $M := (V, \rho)$ is a $KE-$module. In this subsection, we give another decomposition of $\textnormal{Gr}(t, M)$ by considering the \emph{determinant character} of the $t-$dimensional submodules of $M$. We show how one can use this decomposition in order to compute the defining ideal of $\textnormal{Gr}(t, M)$ in the Pl\"ucker space $\mathbb{P}(\bigwedge^t V)$.\\

Let $1\leq t \leq \chi(1_E)$, where we recall that $\chi$ is the character afforded by a fixed $KE-$module $M = (V, \rho)$. Any element of the $t-$th exterior power $\bigwedge^tV$ of $V$ over $K$ is called a \emph{$t-$vector} of $V$ and those of the form $v_1\wedge\ldots\wedge v_t$ are called \emph{pure}. Any element of $\bigwedge^tV$ can be written as a finite sum of pure $t-$vectors. There is moreover an induced action of $E$ on $\bigwedge^tV$ given by, for all $e\in E$ and for any pure $t-$vector $v_1\wedge\ldots\wedge v_t$ of $V$,
\begin{equation}\label{eq:wedge}
    e\cdot(v_1\wedge\ldots\wedge v_t) := (\rho(e) v_1)\wedge\ldots\wedge (\rho(e)v_t).
\end{equation}
We denote by $\bigwedge^t\rho$ the previous representation of $E$ on $\bigwedge^tV$ and $\bigwedge^tM := (\bigwedge^tV, \bigwedge^t\rho)$ the corresponding $KE-$module. We call it the \emph{$t-$th exterior power} of $M$.

\begin{rem}\label{rem equiv closed embed}
The action defined in \cref{eq:wedge} is so that the \emph{Pl\"ucker embedding} $\text{Gr}(t, V) \hookrightarrow\mathbb{P}(\bigwedge^tV)$ is equivariant with respect to $\rho$ and $\bigwedge^t\rho$. The image of $\textnormal{Gr}(t, V)$ under this closed embedding consists of all lines spanned by a pure $t-$vector in $\bigwedge^tV$. Hence one deduces that $\textnormal{Gr}(t, M)$ is non-empty if and only if $\bigwedge^tM$ admits a $1-$dimensional $KE-$submodule whose underlying $K$-vector space is spanned by a pure $t-$vector. If we let $\bigwedge^tM = \bigoplus_{\mu\in\textnormal{Irr}_K(E)}U_{\mu}^{\oplus g_\mu}$ be an isotypical decomposition of $\bigwedge^tM$, then the latter holds if and only if there exists $\mu\in\textnormal{Irr}^1_K(E)$ such that $g_\mu \neq 0$. Indeed, any $E-$invariant line of $\bigwedge^tV$ lies in one of the isotypical components of weight 1 of $\bigwedge^tM$.
\end{rem}

For any $\eta\in C_{\chi}(t)$, we call the \emph{determinant character} of $\eta$, denoted $\textnormal{det}(\eta)$, the character afforded by the $t-$th exterior power of any $KE-$module affording $\eta$. This is a $1-$dimensional character, constituent of the character $\bigwedge^t\chi$ afforded by $\bigwedge^tM$. Note that two distinct constituents $\eta, \eta'\in C_{\chi}(t)$ of $\chi$ can have the same determinant character. For any $\mu\in\text{Irr}^1_K(E)$, we denote $\text{Gr}_\mu(t, M)\subseteq \text{Gr}(t, M)$ the subset of $t-$dimensional $KE-$submodules of $M$ having determinant character equal to $\mu$. Then we have the decompositions, as sets,
\[ \text{Gr}(t, M) = \bigsqcup_{\mu\in\textnormal{Irr}^1_K(E)} \text{Gr}_\mu(t, M)\]
and for all $\mu\in\textnormal{Irr}^1_K(E)$
\[ \text{Gr}_\mu(t, M) = \bigsqcup_{\eta\in C_{\chi}(t), \textnormal{det}(\eta) = \mu}\text{Gr}(\eta, M).\]

By identifying the Grassmannian variety $\text{Gr}(t, V)$ with its image via the Pl\"ucker embedding,
our \Cref{rem equiv closed embed} tells us that
\[ \text{Gr}_\mu(t, M) \simeq \mathbb{P}(U_{\mu}^{g_{\mu}})\cap \text{Gr}(t, V).\]
In this way, we can define a scheme structure on $\textnormal{Gr}_\mu(t, M)$ turning it into a closed subvariety of $\textnormal{Gr}(t, M)$, allowing us to effictively compute the defining ideal of $\text{Gr}(t, M)$ in the Pl\"ucker space.

\section{Finding intersections with prescribed symmetry}\label{sec:2}
\hspace{\parindent} 
In this section, we show how one can use the method explained in the previous sections to find the defining ideals of intersections of hypersurfaces that are fixed by a linear action of a finite group on their ambient complex projective space.\\

Let $X$ be a complex projective variety in $\mathbb{P}^n_{\mathbb{C}}$ given as an intersection of $t$ hypersurfaces of the same degree $d$. The ideal $I$ defining $X$ is homogeneous and generated by $t$ homogeneous polynomials $f_1,\ldots, f_t\in\mathbb{C}[x_0,\ldots,x_n]$ of common total degree $d$.
Let $G$ be a finite group and let $p\colon E\to G$ be a Schur cover of $G$. Suppose that there exists a faithful linear action of $G$ on $\mathbb{P}^n_{\mathbb{C}}$ preserving $X$, which is not necessarily given (we assume existence without any explicit description). Our strategy to find $I$ consists in classifying all faithful linear actions of $G$ on $\mathbb{P}^n_{\mathbb{C}}$ and in determining, for each class of them, a parametrization of all ideals defining intersections of the same type of $X$, and which are preserved under the provided action of $G$.
Any linear action of $G$ on $\mathbb{P}^n_{\mathbb{C}}$ is given by a projective representation $\overline{\rho}\colon G\to \textnormal{PGL}(\mathbb{C}^{n+1})$, which can be lifted along $p$ to $\rho\colon E\to\textrm{GL}(\mathbb{C}^{n+1})$ making the following commutative diagram with exact rows commute
\begin{equation}\label[Diagram]{diag:maindiagbis}
\begin{tikzcd}
1\arrow[r]&H\arrow[d,"\beta"']\arrow[r,"i"]&E\arrow[d,"\rho"']\arrow[r, "p"]&G\arrow[d, "\overline{\rho}"']\arrow[r]&1\\
1\arrow[r]&\mathbb{C}^{\times}\arrow[r,"\cdot\textnormal{Id}_{\mathbb{C}^{n+1}}"']&\textnormal{GL}(\mathbb{C}^{n+1})\arrow[r,"\pi"']&\textnormal{PGL}(\mathbb{C}^{n+1})\arrow[r]&1
\end{tikzcd}.
\end{equation}

\subsection{From invariant ideals to group algebra modules}\label{subsec:redu}
\hspace{\parindent} In this subsection, we fix a linear action of $G$ on $\mathbb{P}^n_{\mathbb{C}}$ and we assume that $X$ is preserved under this action. We then transform the problem of finding the defining ideal of the intersection $X$ into finding a $t-$dimensional $\mathbb{C}E-$submodule of $R_d$, the $d-$homogeneous part of the polynomial algebra associated to $\mathbb{C}^{n+1}$.\\

We denote by $R_{\bullet} := \bigoplus_{h\geq 0} \mathbb{C}[x_0,\ldots,x_n]_h$ the $\mathbb{Z}-$graded $\mathbb{C}-$algebra of polynomials in $n+1$ variables. Considering \cref{diag:maindiagbis}, the action of $E$ on $\mathbb{C}^{n+1}$ defined by $\rho\colon E\to \textnormal{GL}(\mathbb{C}^{n+1})$ naturally induces, for all $h\geq 0$, a linear action on $R_h$. It is given as follows: for any $h\geq 0$, any $P\in R_h$, any $e\in E$ and any $x\in \mathbb{C}^{n+1}$,
\begin{equation*}\label{eq:hompol}
    (e\cdot P)(x) := P(\rho(e)^{-1}(x)).
\end{equation*} 
It is a well-defined action, because the action of $E$ on $\mathbb{C}^{n+1}$ is linear, which we denote by $\rho_h$. Collecting these actions for all $h\geq 0$ gives $(R_{\bullet}, \rho_{\bullet})$ the structure of a $\mathbb{C}E-$algebra - $R_{\bullet}$ is a $\mathbb{Z}-$graded $\mathbb{C}-$algebra and all of its homogeneous components $R_h$ ($h\geq 0$), equipped with the action $\rho_h$, are $\mathbb{C}E-$modules. The following key result shows how to simplify the search for the ideal $I$ defining $X$ by restricting ourselves into determining a basis for a finite-dimensional vector space.

\begin{propo}\label{propo:invideal}
Let $K$ be a field, let $E$ be a group and let $(R_{\bullet}, \rho_{\bullet})$ be a $\mathbb{Z}-$graded $KE-$algebra. Let $I$ be a homogeneous ideal of $R_{\bullet}$ being finitely generated by $t$ homogeneous elements $r_1,\ldots, r_t\in R_{\bullet}$ of respective degrees $d_1,\ldots,d_t$. We denote by $I_h := I\cap R_h$ the $h-$homogeneous part of $I$. Then, $I$ is invariant for the given action of $E$ on $R_{\bullet}$ if and only if $(I_{d_i}, \rho_{d_i})$ is a $KE-$submodule of $(R_{d_i}, \rho_{d_i})$ for all $i=1,\ldots,t$ (here we use the same notation for the restriction of $\rho_h$ to $I_h$, $h\geq 0$).
\end{propo}

\begin{proof}
First remark that $I= \bigoplus_{h\in\mathbb{Z}} I_h = \sum_{i=1}^t I_{h_i}$ as $R_0-$modules since $I$ is generated by the $t$ homogeneous elements $r_1,\ldots,r_t$. Therefore, if $E\cdot I = I$ (i.e. $I$ is $E-$invariant) then for all $i=1,\ldots,t$, we see that $E\cdot I_{h_i} = E\cdot (I\cap R_{h_i})\subseteq I_{h_i}$ because $E$ acts on $R_{h_i}$. Therefore, $(I_{h_i}, \rho_{h_i})$ is a $\mathbb{C}E-$submodule of $(R_{h_i}, \rho_{h_i})$, for all $i=1,\ldots,t$.\\
Now suppose that for all $i=1,\ldots,t$, the $h_i-$ homogeneous part $I_{h_i}$ of $I$ is $E-$invariant. Since $I$ is generated by  $\bigcup_{i=1}^tI_{h_i}$ as a $R_{\bullet}-$module and $(R_{\bullet}, \rho_{\bullet})$ is a $KE-$module, then $I$ is $E-$invariant.
\end{proof}

Considering \cref{diag:maindiagbis} with $R_d$ instead of $\mathbb{C}^{n+1}$, we know that $\rho_d$ reduces to a unique projective representation of $G$ on $R_d$. By commutativity of \cref{diag:maindiagbis} and \cref{propo:invideal} one sees that $X$ is preserved under the action of $G$ on $\mathbb{P}^{n}_{\mathbb{C}}$ if and only if $(I_d, \rho_d)$ is a $t-$dimensional $\mathbb{C}E-$submodule of $(R_d, \rho_d)$. 

\begin{rem}
    This is where the notion of {\em invariant Grassmannian} previously defined has importance. An $E-$invariant homogeneous ideal $I_d$ in $R_d$ might not be generated by semi-invariant polynomials, i.e. homogeneous polynomials whose $\mathbb{C}-$spans in $R_d$ are themselves $\mathbb{C}E-$modules. Indeed, if the underlying $\mathbb{C}E-$module of $I_d$ is simple of dimension greater than one, then the ideal cannot be generated by such semi-invariant polynomials. This is also why we cannot use the known algorithms from invariant theory for finite groups, as described in \cite{dk15} for instance.
\end{rem}

\subsection{Classification of projectively faithful representations}\label[subsec]{subsec:projsuit}
\hspace{\parindent} In \cref{defin:def2} we define an equivalence relation on the set of linear representations of $E$. In this subsection, we define another equivalence relation, coarser than the previous one. More precisely, we classify $p-$projective linear representations of $E$ having faithful reductions, up to similarity of their respective $p-$reductions to $G$.\\

Recall that we are given from \cref{diag:maindiagbis} a Schur cover $p\colon E\twoheadrightarrow G$ of $G$. 
\begin{defin}\label{def:projsui}
Let $V$ be a finite-dimensional $\mathbb{C}-$vector space. A $p-$projective linear representation $\rho\colon E\to \textnormal{GL}(V)$ is said to be \emph{$p-$projectively faithful} if its $p-$reduction $\bar{\rho}\colon G\to\text{PGL}(V)$ is faithful.
\end{defin}

By commutativity of \cref{diag:maindiagbis}, and by surjectivity of $p$, we see that any $p-$projective representation $\rho$ of $E$ on $V$ is $p-$projectively faithful if and only 
\[\ker(\pi\circ\rho) = \ker(p).\]

 In order to find a $\mathbb{C}E-$module $W$ whose underlying vector space generates the defining ideal $I$ of $X$ (see \cref{subsec:redu}), we start by classifying the $p-$projectively faithful representations of $E$ on $\mathbb{C}^{n+1}$. 

\begin{defin}[\cite{isa76}, Definition (2.26)]\label{def:center}
Let $\chi$ be the $\mathbb{C}-$character of $E$ afforded by a $\mathbb{C}E-$module $(V, \rho)$. Then, the \emph{center of the character} $\chi$ is defined to be 
\[ Z(\chi) := \left\{e\in E \;\Bigg\mid\; \frac{\chi(e)}{\chi(1_E)}\;\textrm{is a root of unity} \right\}.\]
\end{defin}

\begin{propo}\label{propo:centerker}
With the notations of \cref{def:center} and \cref{diag:maindiagbis}, $Z(\chi) = \ker(\pi\circ\rho)$.
\end{propo}

\begin{proof}
According to \cite[Lemma (2.27)]{isa76},
\begin{equation}\label{eq:centerchi}
    Z(\chi) = \left\{ e\in E\mid \rho(e)\in \mathbb{C}^{\times}\textnormal{Id}_V\right\}
\end{equation} 
so the first inclusion $Z(\chi)\subseteq \ker(\pi\circ\rho)$ holds. Now, let $e\in E$ be such that $\pi(\rho(e)) = 1_{\textnormal{PGL}(V)}$. In particular $\rho(e) \in \ker(\pi) = \mathbb{C}^{\times}\textnormal{Id}_V$. Therefore according to \cref{eq:centerchi}, we have that $e\in Z(\chi)$. 
\end{proof}

\begin{coropro}\label{coro:1}
A projective representation $\overline{\rho}\colon G\to \textnormal{PGL}(V)$ is faithful if and only if $Z(\chi) = \ker(p)$ where $\chi$ is afforded by any $p-$lift $\rho\colon E\to\textnormal{GL}(V)$ of $\overline{\rho}$.
\end{coropro}

Using \cref{coro:1}, we are now able to decide whether a linear representation of $E$ is $p-$ projectively faithful or not. In fact, checking $p-$projectivity and the condition of \cref{coro:1} are both possible using only character theory. Now, let us note that if $\overline{\rho}$ and $\overline{\rho}'$ are two similar projective representations of $G$ on $V$, then there exists an automorphism $\mathcal{L}\in\textnormal{Aut}(V)$ such that, for all $g\in G$,
\begin{equation*}\label{eq:projeq}
    \overline{\mathcal{L}}\circ\overline{\rho}(g)\circ\overline{\mathcal{L}}^{-1} = \overline{\rho}'(g).
\end{equation*}
This means that if a projective variety in $\mathbb{P}(V)$ is preserved under $\overline{\rho}$, then after a projective base change, it will also be preserved under $\overline{\rho}'$, and vice-versa. We can therefore consider only one representatives of each similarity class, and thus classify $p-$projectively faithful representations of $E$ up to similarity of their respective $p-$reductions. Using \cref{lem:lem1}, this is equivalent to classifying $p-$projectively faithful representations of $E$ on $\mathbb{C}^{n+1}$ by equivalence modulo $\textnormal{Irr}^1_{\mathbb{C}}(E)$. Let us denote $\text{PFR}(E, G, \mathbb{C}^{n+1})$ a set of representatives of classes in 
\begin{equation*}\label{eq:corr}
    \left\{\textrm{$p-$projectively faithful representations of $E$ on $\mathbb{C}^{n+1}$}\right\}/\left\{\rho \sim \rho' \textrm{ iff } \exists \epsilon\in\textrm{Irr}_{\mathbb{C}}^1(E)\textrm{ s.t. } \chi_{\rho} = \epsilon \chi_{\rho'}\right\}.
\end{equation*}
We ensure that $\text{PFR}(E, G, \mathbb{C}^{n+1})$ is finite since the number of equivalence classes of linear representations of $E$ on $\mathbb{C}^{n+1}$ is actually finite (see \cite[Corollary (2.5)]{isa76}). Note that using character theory one can efficiently compute a set of characters afforded by the representatives of classes in $\text{PFR}(E, G, \mathbb{C}^{n+1})$. However, given this set of characters, the computation of the actual representations (given in matrix form) can be much more challenging (see \cite{repalg} for an algorithm and complexity discussion).

\subsection{Application to the case of K3 surfaces}\label{sec:3}
\hspace{\parindent} In this subsection, we apply the previous theory to compute projective models of K3 surfaces given as complete intersections of hypersurfaces of the same degree.\\

Let $(S, G, G_s)$ be the triple given in \Cref{tab:my_labelbis}. From \cite[77b]{bh20}, we know that the canonical $G$-invariant polarization $L$ of the pair $(S, G)$ has degree $c_1(L)^2 = 8$. According to a remark in \cite[Page 615]{sd74}, either the polarization $L$ is \emph{hyperelliptic} and the linear system $|L|$ defines a degree 2 map onto a surface of degree 4 in $\mathbb{P}^5_{\mathbb{C}}$, or it is not hyperelliptic and $S$ is birational to a surface of degree 8 in $\mathbb{P}^5_{\mathbb{C}}$. By \cite[Theorem 5.2.]{sd74}, there are numerical conditions to decide whether such a polarization is hyperelliptic, and such conditions can be checked in practice using an algorithm of Shimada \cite[Algorithm 2.2]{shimada}.
In a recent database of Brandhorst and Hofmann \cite[77.2.1.3]{st}, one can recover workable lattice data about such a triple $(S,G, G_s)$ and show that the $G-$invariant polarization does not satisfy any of the conditions of \cite[Theorem 5.2.]{sd74}. Therefore $L$ is not hyperelliptic and the linear system $|L|$ defines a birational map $\varphi_{\mid L\mid}$ from $S$ onto a surface of degree 8 in $\mathbb{P}^5_{\mathbb{C}}$. We moreover know that $L$ lies in $2\textnormal{NS}(S)^\vee$: according to \cite[Theorem 7.2.]{sd74}, this implies that $\varphi_{\mid L\mid}(S)\subseteq \mathbb{P}^5_{\mathbb{C}}$ is given as a complete intersection of 3 quadrics. Finally, the orthogonal complement of $L$ in $\textnormal{NS}(S)$ contains no vector of square $-2$, so the previous complete intersection is actually smooth. The surface $S$ can therefore be described as a smooth complete intersection of 3 quadrics in $\mathbb{P}^5_\mathbb{C}$ (see also \cite[Theorem 1.3]{deg22}).

Let $G$ be the group with Id $[384, 5602]$ (in the Small Group Library \cite{sgl}). Using GAP \cite{GAP}, one can show that this group has Schur multiplier $M(G)$ isomorphic to $C_2^3$, and therefore any Schur cover of $G$ has order 3072. We compute such a Schur cover $p\colon E\twoheadrightarrow G$, using for instance the GAP method \textsc{EpimorphismSchurCover}: the following steps may differ depending on the choice of the Schur cover, yet the final result shall remain true.
The Schur cover $E$ chosen for these computations has 10 classes of $p-$projectively faithful representations on $F^6$, where $F := \mathbb{Q}(\zeta_{24})$ is the $24-$th cyclotomic field with 24 being the exponent of $E$. Here we work over $F$ instead of $\mathbb{C}$ for computational reasons: according to \cite[Corollary (9.15)]{isa76}, the field $F$ is a splitting field for $E$, so we are allowed to restrict to $F$ (the results remain true over $\mathbb{C}$). In what follows, we denote $z := \zeta_{24}$ and $i := z^6$.

Let $M$ be the $FE-$module $(F^6, \rho)$, where $\rho$ is given by the $\sigma_i$'s in \cref{theo:mainth}. The representation $\rho$ is $p-$projectively faithful, and $N := (R_2, \rho_2)$ is a $21-$dimensional $FE-$module (where $(R_{\bullet}, \rho_{\bullet})$ is defined as in \cref{subsec:redu}). Let $\chi$ be the $F-$character of $E$ afforded by $N$. One has that $C_{\chi}(3) = \{\mu\}$ where $\mu\in\textnormal{Irr}^3_{F}(E)$ and $\langle\chi,\mu\rangle = 2$. Therefore, we have that $\text{Gr}(3, N)$ is irreducible of dimension 1, equal to $\text{Gr}(\mu, N)$. Let $W$ be the isotypical component of $N$ affording $2\mu$. The $FE-$module $W$ consists of the sum of two equivalent simple modules, affording $\mu$, with respective $F-$bases
\[ \left\{w_1 := \left(\begin{array}{c}ix_0x_1 + x_0x_2+x_1x_3+ix_2x_3\\ix_0x_1-x_0x_2-x_1x_3+ix_2x_3\\-x_0x_3-x_1x_2\end{array}\right), w_2 := \left(\begin{array}{c}x_5^2\\x_4^2\\-x_4x_5\end{array}\right)\right\}\]
where $(x_0,\ldots,x_5)$ is a basis for the dual space of $F^6$. Note that these bases are chosen in such a way that the actions of $E$ on each of them have the same matrix representations. We know that any $3-$dimensional submodule of $N$ is then generated by a linear combination of $w_1$ and $w_2$ (\cref{theo:ge}). However, it is easy to see that the ideals respectively generated by $w_1$ and $w_2$ do not define smooth varieties. Now let $\lambda\in \mathbb{C}^{\times}$. The ideal generated by $w_1 + \lambda w_2$ defines a variety $S_{\lambda}$ which is by construction a K3 surface. Moreover, for distinct non-zero $\lambda_1\neq \lambda_2$, the projective change of coordinates
\[[x_0:x_1:x_2:x_3:x_4:x_5]\mapsto \left[x_0:x_1:x_2:x_3:\sqrt{\frac{\lambda_2}{\lambda_1}}x_4:\sqrt{\frac{\lambda_2}{\lambda_1}}x_5\right]\]
commutes with the action of $G$ on $\mathbb{P}^5_\mathbb{C}$ and it maps $S_{\lambda_2}$ to $S_{\lambda_1}$, which are therefore $G-$equivariantly isomorphic K3 surfaces. In this case, we say that $\{S_{\lambda}\}_{\lambda \in \mathbb{C}^{\times}}$  is a \emph{$1-$dimensional $G-$isotrivial family}. As expected, up to equivariance, this K3 surface is unique and given by $S := S_1$ in \cref{theo:mainth}. Finally, to ensure that $S$ corresponds to the case 77b in \cite{bh20}, we need that the subgroup $G_s$ of automorphisms of $G$ acting symplectically on $S$ is isomorphic to the group $T_{192}$. We use the following lemma, which is a direct consequence of the mentioned result of Mukai:

\begin{lem}[\cite{muk88}, Lemma (2.1)]\label{lem:muk}
Let $X\subseteq \mathbb{P}^n_{\mathbb{C}}$ be a smooth complete intersection of $t$ hypersurfaces of degree $d$ such that $td = n+1$. Suppose that $X$ is preserved under a faithful linear action of a finite group $G$ on $\mathbb{P}^n_{\mathbb{C}}$ and let $p\colon E\twoheadrightarrow G$ be a Schur cover. Let $\rho\colon E\to\textnormal{GL}_{n+1}(\mathbb{C})$ be $p-$projectively faithful, such that there exists a $t-$dimensional $\mathbb{C}E-$submodule $M$ of $(R_d, \rho_d)$ generating the ideal defining $X$. Then, the normal subgroup $G_s$ of $G$ consisting of elements whose induced action on $H^{2,0}(X)$ is trivial is given by
\[G_s = p\left(\left\{e\in E\;\mid\; \textnormal{det}(\chi_{\rho})(e) = \textnormal{det}(\chi_M)(e)\right\}\right).\]
\end{lem}

\cref{lem:muk} offers a practical, and computationally feasible, way to compute $G_s$. Here, we can apply it to the group generated by the $\sigma_i$'s on $S_{\lambda}$. One finds that $G_s\cong T_{192}$ with, in particular, $\sigma_i$ acting symplectically on $S_{\lambda}$ for $i=1,2,3,4$ and $\sigma_5$ being a non-symplectic involution. This concludes the proof of \cref{theo:mainth}. See \cref{ap} for a few more examples of K3 surfaces of degree 8 computed in a similar way.

\section{Some geometric comments}\label{sec:sec4}
\hspace{\parindent} In this last section, we make some geometric comments about the pair $(S, G)$ (\cref{sec:3}). In particular, starting from the K3 surface $S$, we compute projective models of a new K3 surface and of a \emph{hyperk\"ahler manifold}\footnote{hyperk\"ahler manifolds, also called \emph{irreducible holomorphic symplectic manifolds}, are higher dimensional analogues of K3 surfaces - see \cite{debarre} for an introduction.}. For each of them, we inspect to which extent the group $G$, acting on the K3 surface $S$, acts on the new varieties constructed.

\subsection{Symplectic quotient}
\hspace{\parindent} The following observations were made with the help of Benedetta Piroddi, who notably pointed out the work in \cite{nikinvo}.\\

As already mentioned, the family of K3 surfaces $\{S_{\lambda}\}_{\lambda\in \mathbb{C}^{\times}}$ obtained in \cref{sec:3} is not a family in the sense of moduli for polarized K3 surfaces. Indeed, every two elements of this family are isomorphic to each other, and such an isomorphism can be made equivariant with respect to the prescribed $G-$action. Therefore, in moduli, this isotrivial family is just a point.

\begin{rem}
    This can already be seen by the rank of the Picard lattice: for $\lambda\neq 0$, the K3 surface $S_{\lambda}$ has Picard rank 20.
\end{rem}

Therefore, we do not have any degeneration of this family inside the associated moduli space. What would then happen at the limit points of this family ?\\
In the case where $\lambda = 0$, we obtain the union of 8 copies $P_1, \ldots, P_8$ of $\mathbb{P}^2$, pairwise meeting at the same rational line $l := V(x_0, x_1, x_2, x_3)\subseteq \mathbb{P}^5$. On the other side, when $\lambda = \infty$, we obtain a non-reduced variety $Z = V(x_4^2, x_5^2, x_4x_5)$ with reduced structure $Z_{red}\simeq \mathbb{P}^3$. Now, both of these extremal cases are pointwise fixed under the involution $\sigma := \sigma_3^2(\sigma_1\sigma_3)^2$ (see \cref{theo:mainth}). Therefore, the 8 points $p_1, \ldots, p_8\in\mathbb{P}^5$ obtained by intersecting each of the $P_i$'s with $Z_{red}$ are fixed under $\sigma$ too. Each of these points lies on $S_{\lambda}$ for all $\lambda\in\mathbb{C}^{\times}$. Since $\sigma$ acts symplectically on any of the $S_{\lambda}$'s, it only fixes points and for all $\lambda\in\mathbb{C}^\times$, the $\sigma-$fixed points of $S_{\lambda}$ are exactly $p_1, \ldots, p_8$.

\begin{rem}
    Note that the choice of $\sigma$ here is not arbitrary: the class $[\sigma]\in\text{PGL}_6(\mathbb{C})$ generates the unique order 2 normal subgroup of $G_s$ (defined as in \cref{lem:muk}).
\end{rem}

Via this observation, and the shape of the equations defining $S_{\lambda}$, one may notice that we fit in the case $\mathcal{M}_{\tilde{8}}$ of \cite[\S 3.7]{nikinvo}. According to the authors, each of the quotient varieties $S_{\lambda}\slash\sigma$ maps, by projection from the $\sigma-$invariant line $l$ onto $Z_{red}$, to the quartic surface $Y \subseteq Z_{red}$ defined by the equation
\[ (z_0z_1+z_2z_3)^2 + (z_0z_2 + z_1z_3)^2 +(z_0z_3+z_1z_2)^2 = 0.\]

This is a nodal surface with 8 singularities, which are respectively the images of the 8 $\sigma-$fixed points $p_i$ under the projection. The resolution of these singularities gives rise to a degree 4 polarized K3 surface $\tilde{Y}\to Y$. Since the subgroup generated by $\sigma$ is normal both in $G_s$ and $G$, the K3 surface $\tilde{Y}$ carries an action of the finite group $G/\sigma \cong C_2^4\rtimes D_6$ with normal symplectic sub-action given by the subgroup $C_2^2\rtimes S_4$.

\begin{rem}
    The finite group $H := G/\sigma$ of order 192 is abstractly isomorphic to the group $H_{192}$ in Mukai's list of maximal symplectic actions on K3 surfaces \cite{muk88}. However, its normal subgroup $H_s$ consisting of symplectic automorphisms has order 96: it is not maximal (in the sense of Mukai), and it corresponds to the group $\#65$ in Xiao's list \cite[Table 2]{xia96}. 
\end{rem}

\subsection{A symmetric hyperk\"ahler fourfold}
\hspace{\parindent} This second part of comments follows from suggestions of Enrico Fatighenti during a poster presentation of this project. Starting from a construction of quadric bundles \cite{beauville}, we use an already known construction to obtain a hyperk\"ahler fourfold from our K3 surface $S$. We then show that the induced action of $G$ on this new variety coincide with its natural induced action as described in \cite{boissiere1}\\

Let $V_6$ be a $6-$dimensional $\mathbb{C}-$vector space such that $S\subseteq \mathbb{P}(V_6)\simeq \mathbb{P}^5$ and let $V_3$ be a $3-$dimensional complex vector space. We see the space $Q_S$ of global quadric section on $S$, generated by
\begin{align*} &q_1 := i x_0x_1+ x_0x_2+ x_1x_3+i x_2x_3+ x_5^2,\\
&q_2 := i x_0x_1- x_0x_2- x_1x_3+i x_2x_3+ x_4^2, \textnormal{ and}\\
&q_3 := - x_0x_3-x_1x_2- x_4x_5,
\end{align*}
as the image of an injective linear map $q\colon V_3\hookrightarrow \textnormal{Sym}^2V_6^{\vee}$. The map $q$ induces an isomorphism $\mathbb{P}(Q_S)\simeq \mathbb{P}(V_3)$ which sends any class of non-zero quadrics on $S$ to the class of coordinates of one representative in the basis $\{q_1, q_2, q_3\}$. We see that through this description, we have a quadric bundle (see \cite[D\'efinition 1.1]{beauville}) $f\colon X \to \mathbb{P}(V_3)$ whose fiber $X_v$ over $[v]\in \mathbb{P}(V_3)$ has projective model $V([q(v)])\subseteq \mathbb{P}(V_6)$. According to \cite[Proposition 1.2]{beauville}, $f$ is a flat morphism whose general fiber is a smooth quadric fourfold. 

\begin{rem}
    The bundle $f$ has singular fibers over a curve $C\subseteq \mathbb{P}(V_3)$ defined by the zero locus of the discriminant form
\[ \Delta([v]) = \textnormal{disc}([q(v)])\]
where we see $q(v)$ as a complex quadratic form on $V_6$. This curve $C$ is of degree 6 with at most nodal singularities (\cite[Proposition 1.2]{beauville}). In our case, for a system of coordinates $\{y_1, y_2, y_3\}$ on $\mathbb{P}(V_3)$, one can check that $C$ is defined by
\[4y_1^5y_2-4y_1^4y_3^2+8y_1^3y_2^3+4y_1y_2^5-7y_1y_2y_3^4-4y_2^4y_3^2-y_3^6 = 0\]
and has 14 nodal points. The double cover of $\mathbb{P}(V_3)$ branched over $C$ is a nodal K3 surface, whose resolution $\tilde{S}$ is therefore a polarized K3 surface of genus 2 and Picard rank at least 15.
\end{rem} 

The map $q\colon V_3\hookrightarrow \textnormal{Sym}^2V_6^{\vee}$ also corresponds to a global section
\[s:= y_1q_1+y_2q_2+y_3q_3\in V_3^{\vee}\otimes \text{Sym}^2V_6^{\vee} \simeq H^0\left(\mathbb{P}(V_3)\times \mathbb{P}(V_6), \mathcal{O}(1,2)\right).\]
In this setting, one can view the quadric bundle $f$ previously defined as the projection from the Fano sixfold $Y$ defined by $s$,
\begin{equation}\label{eq:fano}
    Y := V(y_1q_1+y_2q_2+y_3q_3)\subseteq \mathbb{P}(V_3)\times \mathbb{P}(V_6),
\end{equation} 
onto the first factor. The construction in \cref{eq:fano} is known to the experts as a systematic way of producing examples of Fano varieties of K3 type given a general K3 surface of degree 8 (see for instance \cite[\S 4.1]{enrico1}, and \cite[\S 2.2]{enrico1} for a definition of \emph{K3 type}).

\begin{rem}[{{\cite{enrico}}}]
The fact that $Y$ in \cref{eq:fano} is of K3 type can been proved independently for the Hodge-theoretical sense and for the categorical sense (\cite[Propositions 48 \& 49]{manivel}).
\end{rem}

Now, the space $V_3^{\vee}\otimes \text{Sym}^2V_6^{\vee}$ is also isomorphic to the set of global sections of the homogeneous bundle $\mathcal{U}_{\text{Gr}(2, V_3)}^{\vee}\boxtimes \text{Sym}^2\mathcal{U}_{\text{Gr}(2, V_6)}^{\vee}$ over the product variety $\text{Gr}(2, V_3)\times \text{Gr}(2, V_6)$ (where $\mathcal{U}$ denotes the tautological bundle). To make a clear distinction, we denote $\tilde{s}\in V_3^{\vee}\otimes \text{Sym}^2V_6^{\vee}$ to be the same as $s$ but seen as a global section of $\mathcal{U}_{\text{Gr}(2, V_3)}^{\vee}\boxtimes \text{Sym}^2\mathcal{U}_{\text{Gr}(2, V_6)}^{\vee}$. The upshot here is the following: according to \cite[Proposition 3.1.3]{benedetti}, the zero locus $\tilde{Y}$ of $\tilde{s}$ inside $\text{Gr}(2, V_3)\times \text{Gr}(2, V_6)$ is actually isomorphic to the Hilbert scheme $S^{[2]}$ of length 2 subschemes of the initial K3 surface $S$. The variety $\tilde{Y}$ is therefore a hyperk\"ahler fourfold of $\textnormal{K3}^{[2]}-$type (see \cite{debarre} for definitions).

\begin{propo}
    The finite group $G$ of automorphisms of $S$ acts faithfully on $\tilde{Y}$ with symplectic sub-action also given by $G_s$.
\end{propo}

\begin{proof}
     The surface $S$ being preserved under the action of $G$ on $\mathbb{P}(V_6)$, we have an induced action of $G$ on $\mathbb{P}(Q_S)$ (given as in the beginning of \cref{subsec:redu}). Since the action of $G$ on $\mathbb{P}(V_6)$ is faithful, we have that
    \[g\cdot([v], [x]) = (q^{-1}(g\cdot [q(v)]), g\cdot [x]),\;\; \forall g\in G,\;\; \forall ([v], [x])\in \mathbb{P}(V_3)\times \mathbb{P}(V_6)\]
    is a well-defined faithful action. This action induces an action of $G$ on $\mathbb{P}(\bigwedge^2V_3)\times \mathbb{P}(\bigwedge^2V_6)$ which restricts to a faithful action of $G$ on the product variety $\text{Gr}(2, V_3)\times \text{Gr}(2, V_6)$ (see \cref{eq:wedge}). Following the proof of \cite[Proposition 3.1.3]{benedetti}, if $(A, B)\in V(\tilde{s})$, then the composite map
    \[ \Phi_{B}\colon V_3\overset{q}{\to} \textnormal{Sym}^2V_6^{\vee}\to \textnormal{Sym}^2B^{\vee}\]
    has kernel equal to $A$ and the zero locus of its image defines a length 2 subscheme of $S$: this gives rise to the isomorphism $\varphi$ between $V(\tilde{s})$ and $S^{[2]}$. We aim to prove that for all $g\in G$, the action of $g$ on $\text{Gr}(2, V_3)\times \text{Gr}(2, V_6)$ preserves $V(\tilde{s})$ and its induced action via $\varphi$ on $S^{[2]}$ coincides with the natural induced action $g^{[2]}$ of $g$ on $S^{[2]}$ (see \cite[D\'efinition 1]{boissiere1} for a definition).\\
    
    First of all if $(A, B)\in V(\tilde{s})$, by definition of $\tilde{s}$ we have that for all $a\in A$ and for all $b\in B$, $q(a)(b) = 0$.
    Therefore for all $g\in G$, $a\in A$ and $b\in B$, since $G$ acts faithfully on $\mathbb{P}(V_6)$,
    \[q(g\cdot a)(g\cdot b) = (g\cdot q(a))(g\cdot b) = q(a)(g^{-1}\cdot (g\cdot(b))) = q(a)(b) = 0.\]
    Hence $(g\cdot A, g\cdot B)\in V(\tilde{s})$ and $V(\tilde{s})$ is preserved by the faithful action of $G$ on $\text{Gr}(2, V_3)\times \text{Gr}(2, V_6)$. Now, one can show that if $g\in G$, $B\in \text{Gr}(2, V_6)$ and $v\in V_3$, then as quadratic forms on $g\cdot B$,
    \[g\cdot \Phi_{B}(v) = \Phi_{g\cdot B}(g\cdot v).\]
    So, if $S \supseteq Z = \varphi((A, B))$ is a length 2 subscheme, we have that $Z = V(\textnormal{Im}(\Phi_B))$ and the action of $G$ on $Z$ being induced by the action of $G$ on $\mathbb{P}(V_6)$, we have that for all $g\in G$ and for all $Q\in \textnormal{Im}(\Phi_{g\cdot B})$
    \[Q(g^{[2]}(Z)) = (g^{-1}\cdot Q)(Z) = 0\]
    since $g^{-1}\cdot Q\in \textnormal{Im}(\Phi_B)$. Therefore $g^{[2]}(Z) = \varphi((g\cdot A, g\cdot B))$ as expected. We conclude thanks to \cite[Theorem 1]{boissiere} which tells us that symplectic automorphisms on $S$ induce naturally symplectic automorphisms on $S^{[2]}$.
\end{proof}

We end this section with an interesting remark of \cite{enrico} about the manifold $\tilde{Y}$. Let us consider the first projection $\tilde{f}\colon \tilde{Y}\twoheadrightarrow \textnormal{Gr}(2, V_3)$. According to \cite[Remark 3.12]{debarre}, since $(S, L)$ is a general K3 surface of degree 8 (see notation of \cref{sec:3}), the pullback $\tilde{f}^{\ast}\mathcal{O}_{\text{Gr}(2, V_3)}$ is primitive and is isotropic for the BBF-form on $\tilde{Y}$. Therefore, \cite[Theorem 1.3]{dhmv} tells us that under these conditions $\tilde{f}$ defines a Lagrangian fibration. In our case, one can actually get an explicit description of the general fiber. Let $A\in \textnormal{Gr}(2, V_3)$ be general. The fiber over $A$ consists of all $2-$spaces $B$ in $V_6$ such that $\Phi_B$ vanishes exactly at $A$. Denoting $Q$ and $Q'$ two quadratic forms spanning $q(A)$ over $\mathbb{C}$, we have that $\tilde{f}^{-1}(A)$ consists of all the $2-$spaces in $V_6$ which are maximal isotropic with respect to both $Q$ and $Q'$ simultaneously. Therefore, seeing both $Q$ and $Q'$ as sections of the bundle $\textnormal{Sym}^2\mathcal{U}_{\textnormal{Gr}(2, V_6)}^{\vee}$ over $\textnormal{Gr}(2, V_6)$, we can identity $\tilde{f}^{-1}(A)$ with the intersection $\tilde{A} := V(Q)\cap V(Q')\subseteq \textnormal{Gr}(2, V_6)$.

\begin{rem}[{{\cite{enrico}}}]
    $\tilde{A}$ is often called a \emph{doubly orthogonal Grassmannian} and denoted $\textnormal{OGr}_2(2, V_6)$ (both $V(Q)$ and $V(Q')$ define copies of the \emph{orthogonal Grassmannian} $\textnormal{OGr}(2, V_6)\subseteq \textnormal{Gr}(2, V_6)$).
\end{rem}

This variety $\tilde{A}$ is actually an abelian surface: indeed, considering now $Q$ and $Q'$ as quadric sections on $\mathbb{P}(V_6)$, the variety $\tilde{A}$ can be seen as the variety of lines on the complete intersection $V(Q, Q')\subseteq \mathbb{P}(V_6)$. As shown in \cite[\S 3]{miles}, $\tilde{A}$ is thus isomorphic to the Jacobian variety of the curve obtained by taking the double cover of $\mathbb{P}(q(A))$ branched along the 6 classes of singular quadrics in $q(A)$.

\begin{rem}
 The fact that there are only 6 classes of singular quadrics in $q(A)$ follows from the fact that $A$ is chosen to be general and the discriminant form on $\text{Sym}^2V_6^{\vee}$ is homogeneous of degree 6. Therefore, the divisor in $\mathbb{P}(\text{Sym}^2V_6^{\vee})$ parametrising singular complex quadratic forms on $V_6$ cuts $\mathbb{P}(q(A))$ in exactly 6 points.
\end{rem}

\appendix 
\section{More polarized K3 surfaces of degree 8}\label{ap}
\hspace{\parindent} In \cref{tab:tabk3}, we provide equations for different triples $(S, G, G_s)$ consisting of a polarized K3 surface $S$ of degree 8, a finite group $G$ of automorphisms of $S$, with symplectic sub-action given by $G_s$. Using the numerical data available in the database of \cite{st}, together with the theory of \cite{sd74} and an algorithm from \cite{shimada}, we know that there exist such symmetric K3 surfaces which can be described as complete intersections of 3 quadrics in $\mathbb{P}^5_\mathbb{C}$. Note that now, the corresponding K3 surfaces with such symmetries are not necessarily unique, and there are also several deformation families with the same group actions and degree. Hence \cref{tab:tabk3} is not complete.

In \cref{tab:tabk3}, the groups $G$ and $G_s$ are identified by their ID's in the Small Group Library \cite{sgl}, the column $\#$ gives the entry of \cite[Table 2]{xia96} corresponding the groups $G_s$, and any $\zeta_n$ denotes a primitive $n$-th root of unity. Whenever it makes sense, we give parameters $\alpha_i$'s in the equations, with $1\leq i\leq D$, arising from the output of the algorithm explained in the paper (after removing parameters describing isotrivial families). Finally, we also tell whether the generic element in the family described by the associated equations is  smooth.

\renewcommand{\arraystretch}{1.3}
\begin{table}[t!]
    \resizebox{\textwidth}{!}{\small{\begin{tabular}{|c|c|c|c|c|c|}
        \hline
           $S$&$G$&$G_s$&$\#$&$D$&smooth\\
        \hline
           $\left\{\begin{array}{lll}
          x_0^2+x_1^2+x_2^2+x_3^2+x_4^2+\alpha_1x_5^2&=&0\\
          -\zeta_{10}^3x_0^2-\zeta_{10}x_1^2+\zeta_{10}^4x_2^2+\zeta_{10}^2x_3^2+x_4^2&=&0\\
          -\zeta_{10}x_0^2+\zeta_{10}^2x_1^2-\zeta_{10}^3x_2^2+\zeta_{10}^4x_3^2+x_4^2&=&0\end{array}\right.$&$[160, 235]$&$[16,14]$&21&1&yes\\
        \hline
        $\left\{\begin{array}{lll}
        x_0x_1+x_2x_3+x_4x_5&=&0\\
        (1-\zeta_4)x_0^2+(\zeta_4-1)x_1^2+\zeta_4x_2^2-\zeta_4x_3^2-x_4^2+x_5^2&=&0\\
        \zeta_4x_0^2-\zeta_4x_1^2+(1-\zeta_4)x_2^2+(\zeta_4-1)x_3^2-x_4^2+x_5^2&=&0
        \end{array}\right.$&$[96, 226]$&$[48, 48]$&51&0&yes\\
        \hline
     
        $\left\{\begin{array}{lll}
        x_0x_1+x_2^2+x_3^2+x_4^2+x_5^2&=&0\\
        \zeta_4x_0^2-\zeta_4x_1^2-\alpha_1(x_4^2-x_5^2)&=&0\\
        x_0^2+x_1^2-\alpha_1(x_2^2-x_3^2)&=&0
        \end{array}\right.$&$[128, 928]$&$[64, 138]$&56&1&yes\\
        
        \hline
        
        $\left\{\begin{array}{lll}
        x_0^2+\alpha_1x_1x_2-\zeta_6x_3^2+(\zeta_6-1)x_4^2+x_5^2&=&0\\
        x_1^2+\alpha_1x_0x_2+x_3^2+x_4^2+x_5^2&=&0\\
        x_2^2+\alpha_1x_0x_1+(\zeta_6-1)x_3^2-\zeta_6x_4^2+x_5^2&=&0
        \end{array}\right.$&$[144,189]$&$[72, 43]$&61&1&yes\\
       
        
        \hline
        
         $\left\{\begin{array}{lll}
          x_0^2+x_1^2+x_2^2-x_3^2-x_4^2-x_5^2&=&0\\
          \zeta_{3}x_0^2-(1+\zeta_{3})x_1^2+x_2^2-\alpha_1((1+\zeta_{3})x_3^2-\zeta_3x_4^2-x_5^2)&=&0\\
          \zeta_{3}x_3^2-(1+\zeta_{3})x_4^2+x_5^2-\alpha_1((1+\zeta_{3})x_0^2-\zeta_3x_1^2-x_2^2)&=&0\end{array}\right.$&$[192,1538]$&$[96,227]$&65&1&yes\\
        
        \hline
        
        $\left\{\begin{array}{lll}
        x_0^2+x_1^2+x_2^2+x_3^2+x_4^2+x_5^2&=&0\\
        (\zeta_6-1)x_0^2+(\zeta_6-1)x_1^2-\zeta_6x_2^2-\zeta_6x_3^2+x_4^2+x_5^2&=&0\\
        -\zeta_6x_0^2-\zeta_6x_1^2+(\zeta_6-1)x_2^2+(\zeta_6-1)x_3^2+x_4^2+x_5^2&=&0
        \end{array}\right.$&$[384, 18235]$&$[192, 1023]$&75&0&no\\
        
        \hline
        
        $\left\{\begin{array}{lll}
        -x_0^2+x_1^2-\alpha_1(x_4^2+x_5^2)&=&0\\
        -x_2^2+x_3^2-\alpha_1(x_0^2+x_1^2)&=&0\\
        -x_4^2+x_5^2-\alpha_1(x_2^2+x_3^2)&=&0
        \end{array}\right.$&$[384, 18235]$&$[192, 1023]$&75&1&yes\\
        
        \hline
        
        $\left\{\begin{array}{lll}
        x_0^2+\zeta_3^2x_3^2+\zeta_3x_4^2+x_5^2&=&0\\
        x_1^2+\zeta_3x_3^2+\zeta_3^2x_4^2+x_5^2&=&0\\
        x_2^2+x_3^2+x_4^2+x_5^2&=&0
        \end{array}\right.$&$[576, 8657]$&$[288, 1026]$&78&0&yes\\
        \hline
    \end{tabular}}}
    \caption{Equations of algebraic K3 surfaces of degree 8}
    \label{tab:tabk3}
\end{table}

\newcommand{\etalchar}[1]{$^{#1}$}

\end{document}